\documentclass[12pt]{amsart}

\usepackage{amssymb}
\usepackage{bm}
\usepackage{graphicx}
\usepackage{psfrag}
\usepackage{color}
\usepackage{url}
\usepackage{algpseudocode}
\usepackage{fancyhdr}
\usepackage{xy}
\input xy
\xyoption{all}

\usepackage{hyperref}
\hypersetup{
	colorlinks=true,
	linkcolor=blue,
	filecolor=magenta,      
	urlcolor=cyan,
}

\newtheorem*{maintheorem*}{Main Theorem}
\newtheorem{theorem}{Theorem}[section]
\newtheorem{prop}[theorem]{Proposition}

\newtheorem{cor}[theorem]{Corollary}

\theoremstyle{definition}
\newtheorem{definition}[theorem]{Definition}
\newtheorem{example}[theorem]{Example}
\numberwithin{equation}{section}


\begin{document}
	
	\mbox{}
	\title{On the local k-elasticities of Puiseux monoids}
	\author{Marly Gotti}
	\address{Department of Mathematics\\University of Florida\\Gainesville, FL 32611}
	\email{marlycormar@ufl.edu}
	\date{\today}
	
	\begin{abstract}
		If $M$ is an atomic monoid and $x$ is a nonzero non-unit element of $M$, then the set of lengths $\mathsf{L}(x)$ of $x$ is the set of all possible lengths of factorizations of $x$, where the length of a factorization is the number of irreducible factors (counting repetitions). In a recent paper, F.~Gotti and C.~O'Neil studied the sets of elasticities $\mathcal{R}(P) := \{\sup \mathsf{L}(x)/\inf \mathsf{L}(x) : x \in P\}$ of Puiseux monoids $P$. Here we take this study a step further and explore the local $k$-elasticities of the same class of monoids. We find conditions under which Puiseux monoids have all their local elasticities finite as well as conditions under which they have infinite local $k$-elasticities for sufficiently large $k$. Finally, we focus our study of the $k$-elasticities on the class of primary Puiseux monoids, proving that they have finite local $k$-elasticities if either they are boundedly generated and do not have any stable atoms or if they do not contain $0$ as a limit point. \\
	\end{abstract}
	\medskip
	
	\maketitle
	
	\section{Introduction} \label{sec:introduction}
	
	Rings of integers of algebraic number fields are not necessarily factorial. We can use the class group of a ring of integers $R$ to measure to which extent elements in $R$ fail to have a unique factorization. In~\cite{lC60}, L.~Carlitz characterized the half-factorial rings of integers in terms of their class groups. A friendly survey illustrating this characterization when $R = \mathbb{Z}[\sqrt{-5}]$ is provided in \cite{CGG17}. After the publication of Carlitz's result, many authors attempted to characterize the class group of a general ring of integers in terms of further arithmetical properties describing the non-uniqueness of factorizations in such a ring (Rush \cite{dR83} was the first to give a complete characterization).
	
	More generally, the algebraic invariants of several non-factorial Noetherian domains can be used to understand how far are such domains from being factorial. Because many of the factorization-related questions on integral domains are independent of the ring additive structure, in the last few decades the study of the phenomenon of non-unique factorization has been extended to the setting of atomic monoids. The monograph \cite{GH06b} by A.~Geroldinger and F.~Halter-Koch has significantly influenced the shape of the modern non-unique factorization theory, which considers not only integral domains but also atomic monoids.
	
	The purpose of modern non-unique factorization theory is to measure how far an atomic monoid is from being factorial (or half-factorial). To carry out this measurement, we focus attention on several factorization invariants, which include the system of sets of lengths, the union of sets of lengths (first introduced in \cite{CS98}), and the elasticity. Roughly speaking, the elasticity of an atomic monoid $M$ is given by
	\[
		\rho(M) = \sup \{\rho(x) : x \in M\},
	\]
	where $\rho(x)$ is the quotient of the maximum possible length of factorizations of $x$ by the minimum possible length of factorizations of $x$. Our aim in this paper is to study an $\mathbb{N}$-parametrized local version of the elasticity of Puiseux monoids, which are, up to isomorphism, the additive submonoids of $(\mathbb{Q},+)$ that are not groups.
	
	This paper is organized as follows. In Section~\ref{sec:preliminaries}, we review notation and introduce most of the concepts we shall be using later. In Section~\ref{sec:general case}, we begin our exploration of the local elasticities of atomic Puiseux monoids. In particular, we find conditions under which Puiseux monoids have all their local elasticities finite as well as conditions under which they have infinite local $k$-elasticities for sufficiently large $k$. Lastly, in Section~\ref{sec:primary case}, we target the class of primary Puiseux monoids, proving that they have finite local elasticities if they do not contain stable atoms or if they do not contain $0$ as a limit point.
	\medskip
	
	\section{Definitions \& Notations} \label{sec:preliminaries}
	
	Throughout this paper, we let $\mathbb{N}$ denote the set of positive integers, and we set $\mathbb{N}_0 := \mathbb{N} \cup \{0\}$. For each subset $A$ of $\mathbb{Q}$, we let $A^\bullet$ denote $A \setminus \{0\}$. Also, for all $q \in \mathbb{Q}$ such that $q > 0$, we let $\mathsf{n}(q)$ and $\mathsf{d}(q)$ denote the unique pair of relatively prime positive integers satisfying that $q = \mathsf{n}(q)/\mathsf{d}(q)$. In this case, we call $\mathsf{n}(q)$ and $\mathsf{d}(q)$ the \emph{numerator} and \emph{denominator} of $q$, respectively. Moreover, for a subset $A$ of positive rationals, we set
	\[
		\mathsf{n}(A) := \{\mathsf{n}(a) : a \in A \} \ \text{ and } \ \mathsf{d}(A) := \{\mathsf{d}(a) : a \in A\}. \vspace{2pt}
	\]
	Although most of the definitions given in this section make sense in a much broader context, for the sake of simplicity we will present them in a particular setting which is enough for the treatment of Puiseux monoids. Every monoid here is tacitly assumed to be commutative, cancellative, and reduced (i.e., only the identity is a unit). As we shall be working in a commutative environment, unless otherwise specified we will use additive notation. Let $M$ be a monoid.
	 
	 \begin{definition}
	 	An element $a \in M \setminus \{0\}$ is called an \emph{atom} (i.e., \emph{irreducible}) provided that for all $x,y \in M$ the fact that $a = x+y$ implies that either $x=0$ or $y=0$. Let $\mathcal{A}(M)$ denote the set of all atoms of $M$.
	 \end{definition}
 	\medskip

	If $A$ is a subset of $M$, then the minimal submonoid of $M$ containing $A$ is denoted by $\langle A \rangle$. If $M = \langle A \rangle$, then we say that $M$ is \emph{generated} by $A$ or that $A$ is a \emph{generating set} of $M$. The monoid $M$ is called \emph{finitely generated} provided that it contains a finite generating set.
	
	\begin{definition}
		If $M = \langle \mathcal{A}(M) \rangle$, then we call the monoid $M$ atomic.
	\end{definition}
	\medskip

	Clearly, every generating set of an atomic monoid $M$ contains $\mathcal{A}(M)$. In addition, it is not hard to prove that $M$ is atomic if and only if it contains exactly one minimal generating set, namely $\mathcal{A}(M)$; see, for instance, \cite[Proposition~1.1.7]{GH06b}.
	
	\begin{definition}
		A \emph{Puiseux monoid} is an additive submonoid of $(\mathbb{Q}, +)$ consisting of nonnegative rationals.
	\end{definition}
	\medskip
	
	As we mentioned in the introduction, each additive submonoid of $(\mathbb{Q},+)$ that is not a group is isomorphic to a Puiseux monoid \cite[Theorem~2.9]{rG84}. Puiseux monoids have a fascinating atomic structure. Some of them contain no atoms at all, as it is the case of $\langle 1/2^n : n \in \mathbb{N} \rangle$, while there are others whose sets of atoms are dense in the nonnegative real line \cite[Theorem~3.5]{GGP16}. The atomicity of members of the family of Puiseux monoids has only been recently studied (see \cite{fG17a} and \cite{GG17}).
	
	There are three classes of Puiseux monoids we shall be studying, namely the classes of bounded, strongly bounded, and primary Puiseux monoids. Let $P$ be a Puiseux monoid. We say that $P$ is \emph{bounded} (respectively, \emph{strongly bounded}) if it can be generated by a set $A$ of rational numbers such that $A$ is bounded (respectively, $\mathsf{n}(A)$ is bounded). In addition, $P$ is {\it primary} if it can be generated by a subset of positive rationals whose denominators are pairwise distinct prime numbers.

	Bounded and strongly bounded Puiseux monoids are not necessarily atomic; see, for example, $\langle 1/2^n : n \in \mathbb{N} \rangle$. However, it is not hard to verify that primary monoids are always atomic. Indeed, it was proved in \cite{GG17} that every submonoid of a primary Puiseux monoid is atomic. The class of atomic Puiseux monoids is plentiful as the following theorem indicates.
	
	\begin{theorem} \cite[Theorem~3.10]{GG17}\label{theo:sufficient condition for atomicity}
		Let $P$ be a Puiseux monoid. If $0$ is not a limit point of $P$, then $P$ is atomic.
	\end{theorem}
	\medskip
	
	Given a set $S$, it is not hard to verify that the formal sums of elements of $S$ (up to permutation) is a monoid, which is called the \emph{free commutative monoid} on $S$. For an atomic monoid $M$, we let $\mathsf{Z}(M)$ denote the free commutative monoid on $\mathcal{A}(M)$. The elements of $\mathsf{Z}(M)$ have the form $a_1 + \dots + a_n$ for some $a_1, \dots, a_n \in \mathcal{A}(M)$ and are called \emph{factorizations}. It follows immediately that the function $\phi \colon \mathsf{Z}(M) \to M$ defined by $\phi(a_1 + \dots + a_n) = a_1 + \dots + a_n$ is a monoid homomorphism.
	
	\begin{definition}
		The homomorphism $\phi$ given above is called the \emph{factorization homomorphism} of $M$.
	\end{definition}
	\medskip

	If $x \in M$, then the \emph{set of factorizations} of $x$, denoted by $\mathsf{Z}_M(x)$, is defined to be the preimage of $x$ by $\phi$, i.e.,
	\[
		\mathsf{Z}_M(x) = \phi^{-1}(x) \subseteq \mathsf{Z}(M).
	\]
	It follows that $M$ is atomic if and only if $\mathsf{Z}_M(x)$ is nonempty for all $x \in M$. If $z = a_1 + \dots + a_n$ for some $a_1, \dots, a_n$, then $n$ is called the \emph{length} of $z$ and is denoted by $|z|$. For $x \in M$, the \emph{set of lengths} of $x$ is the set
	\[
		 \mathsf{L}_M(x) := \{|z| : z \in \mathsf{Z}_M(x)\}.
	\]
	We write $\mathsf{Z}(x)$ and $\mathsf{L}(x)$ for the respective sets $\mathsf{Z}_M(x)$ and $\mathsf{L}_M(x)$ when there is no risk of ambiguity. In addition, the collection of sets
	\[
		\mathcal{L}(M) := \{ \mathsf{L}(x): x \in M \}
	\]
	is called the \emph{system of sets of lengths} of $M$. Systems of sets of lengths of many families of atomic monoids have been the focus of a great deal of research during the last few decades (see, for example, \cite{ACHP07,GS16,wS09}).
		
	We proceed to introduce unions of sets of lengths and local elasticities. Similar to the system of sets of lengths, the elasticity is another arithmetical invariant used to measure up to what extent factorizations in monoids (or domains) fail to be unique. The concept of elasticity was introduced by R.~Valenza \cite{rV90} in the context of algebraic number theory. The \emph{elasticity} $\rho(M)$ of an atomic monoid $M$ is given by
	\[
		\rho(M) = \sup \{\rho(x) : x \in M\}, \ \text{where} \ \rho(x) = \frac{\sup \mathsf{L}(x)}{\min \mathsf{L}(x)}.
	\]
	For $n \in \mathbb{N}_0$, we define $\mathsf{L}^{-1}(n) := \{x \in M : n \in \mathsf{L}(x)\}$. Now, the \emph{union of sets of lengths} of $M$ containing $n$ is defined to be
	\[
		\mathcal{U}_n(M) = \{|z| : z \in \mathsf{Z}(x) \ \text{for some} \ x \in \mathsf{L}^{-1}(n) \}.
	\]
	
	\begin{definition}
		The \emph{$n$-th local elasticity} of $M$ is defined by
		\[
			\rho_n(M) = \sup \mathcal{U}_n(M).
		\]
	\end{definition}
	\medskip
	
	A numerical semigroup is a cofinite additive submonoid of $(\mathbb{N}_0, +)$. It is well known that every numerical semigroup is finitely generated and, therefore, atomic \cite[Proposition~2.7.8(4)]{GH06b}. See \cite{GR09} for an introduction to numerical semigroups. For a numerical semigroup $N$ with minimal generating set $A$, it was proved in \cite[Section~2]{CHM06} that the elasticity of $N$ is given by $\max A/ \min A$. On the other hand, it is not hard to verify that $\mathcal{U}_n(N)$ is bounded and, therefore, every local elasticity of $N$ is finite. In the next two sections, we will generalize this fact in two different ways to Puiseux monoids.
	\medskip
	
	\section{The General Case} \label{sec:general case}
	
	We begin this section proposing a sufficient condition under which most of the local elasticities of an atomic Puiseux monoid have infinite cardinality. On the other hand, we describe a subclass of Puiseux monoids (containing isomorphic copies of each numerical semigroup) whose local $k$-elasticities are finite.

	If $P$ is a Puiseux monoid, then we say that $a_0 \in \mathcal{A}(P)$ is \emph{stable} provided that the set $\{a \in \mathcal{A}(P) : \mathsf{n}(a) = \mathsf{n}(a_0)\}$ is infinite.
	
	\begin{prop} \label{prop:union of sets of lengths: infinite case}
		Let $P$ be an atomic Puiseux monoid. If $P$ contains a stable atom, then $\rho_k(P)$ is infinite for all sufficiently large $k$.
	\end{prop}
	
	\begin{proof}
		Suppose that for some $m \in \mathbb{N}$ the set $A := \{a \in \mathcal{A}(P) : \mathsf{n}(a) = m\}$ contains infinitely many elements. Let $\{a_n\}$ be an enumeration of the elements of $A$. Because the elements of $A$ have the same numerator, namely $m$, we can assume that the sequence $\{a_n\}$ is decreasing. Setting $d = \mathsf{d}(a_1)$, we can easily see that $d a_1 = m = \mathsf{d}(a_j) a_j$ for each $j \in \mathbb{N}$. Therefore $\mathsf{d}(a_j) \in \mathcal{U}_d(P)$ for each $j \in \mathbb{N}$. As $\mathsf{d}(A)$ is an infinite set so is $\mathcal{U}_d(P)$. The fact that $|\mathcal{U}_d(P)| = \infty$ immediately implies that $|\mathcal{U}_k(P)| = \infty$ for all $k \ge d$. Hence $\rho_k(P) = \sup \, \mathcal{U}_k(P) = \infty$ for every $k \ge d$.
	\end{proof}
	\medskip
	
	Recall that a Puiseux monoid $P$ is strongly bounded if it can be generated by a set of rationals $A$ whose numerator set $\mathsf{n}(A)$ is bounded. As a direct consequence of Proposition~\ref{prop:union of sets of lengths: infinite case} we obtain the following result.
	
	\begin{cor}
		If $P$ is a non-finitely generated strongly bounded atomic Puiseux monoid, then $\rho_k(P)$ is infinite for all $k$ sufficiently large.
	\end{cor}
	\medskip	
		
	In contrast to the previous proposition, the next result gives a condition under which Puiseux monoids have finite $k$-elasticity for each $k \in \mathbb{N}$.
	
	\begin{prop} \label{prop:union of sets of lengths: finite case}
		Let $P$ be a Puiseux monoid that does not contain $0$ as a limit point. If $P$ is bounded, then $\rho_k(P) < \infty$ for every $k \in \mathbb{N}$.
	\end{prop}
	
	\begin{proof}
		Because $0$ is not a limit point of $P$, it follows by Theorem~\ref{theo:sufficient condition for atomicity} that $P$ is atomic As $P$ is a bounded Puiseux monoid, $\mathcal{A}(P)$ is a bounded set of rational numbers. Take $q, Q \in \mathbb{Q}$ such that $0 < q < a < Q$ for all $a \in \mathcal{A}(P)$. Now fix $k \in \mathbb{N}$, and suppose that $\ell \in \mathcal{U}_k(P)$. Then there exists $x \in \mathsf{L}^{-1}(k)$ such that $\ell \in \mathsf{L}(x)$. Because $x$ has a factorization of length $k$, it follows that $x < kQ$. Taking $a_1, \dots, a_\ell \in \mathcal{A}(P)$ such that $x = a_1 + \dots + a_\ell$, we find that
		\[
			q \ell < a_1 + \dots + a_\ell = x < kQ.
		\]
		Therefore $\ell < kQ/q$. Because neither $q$ nor $Q$ depends on the choice of $x$, one obtains that $\mathcal{U}_k(P)$ is bounded from above by $kQ/q$. Hence $\rho_k(P) = \sup \mathcal{U}_k(P)$ is finite, and the proof follows.
	\end{proof}
	\medskip
	
	With the following two examples, we shall verify that the conditions of containing a stable atom and not having $0$ as a limit point are not superfluous in Proposition~\ref{prop:union of sets of lengths: infinite case} and Proposition~\ref{prop:union of sets of lengths: finite case}, respectively.
	
	\begin{example}
		Let $\{p_n\}$ be a strictly increasing enumeration of the prime numbers, and consider the following Puiseux monoid:
		\[
		P = \langle A \rangle, \ \text{ where } \ A = \bigg\{ \frac{p_n - 1}{p_n} \ : \ n \in \mathbb{N} \bigg \}.
		\]
		As the denominators of elements in $A$ are pairwise distinct primes, it immediately follows that $\mathcal{A}(P) = A$. Therefore $P$ is atomic. Clearly, $P$ does not contain stable atoms. Because $A$ is bounded so is $P$ (as a Puiseux monoid). On the other hand, $0$ is not a limit point of $P$. Thus, it follows by Proposition~\ref{prop:union of sets of lengths: finite case} that $\rho_k(P)$ is finite for every $k \in \mathbb{N}$.
		Notice also that
		\begin{enumerate}
			\item if $q \in P$ has at least two factorizations with no atoms in common, then $q \in \mathbb{N}$;
			\item by Proposition \ref{prop:union of sets of lengths: finite case}, we have both a lower and an upper bound for any $q \in ~\mathsf{L}^{-1}(k)$.
		\end{enumerate}

		Using the previous two observations, we have created an 
		\href{https://www.github.com/marlycormar/find\_u\_k}{R-script} that generates the sets $U_k$ for $k \in \{1, \dots, 15\}$. Each $U_k$ appears as the $k$-th column in Table~\ref{fig:U_k}.
		
		\begin{table}[h!]
			\centering
			\includegraphics[scale=.45]{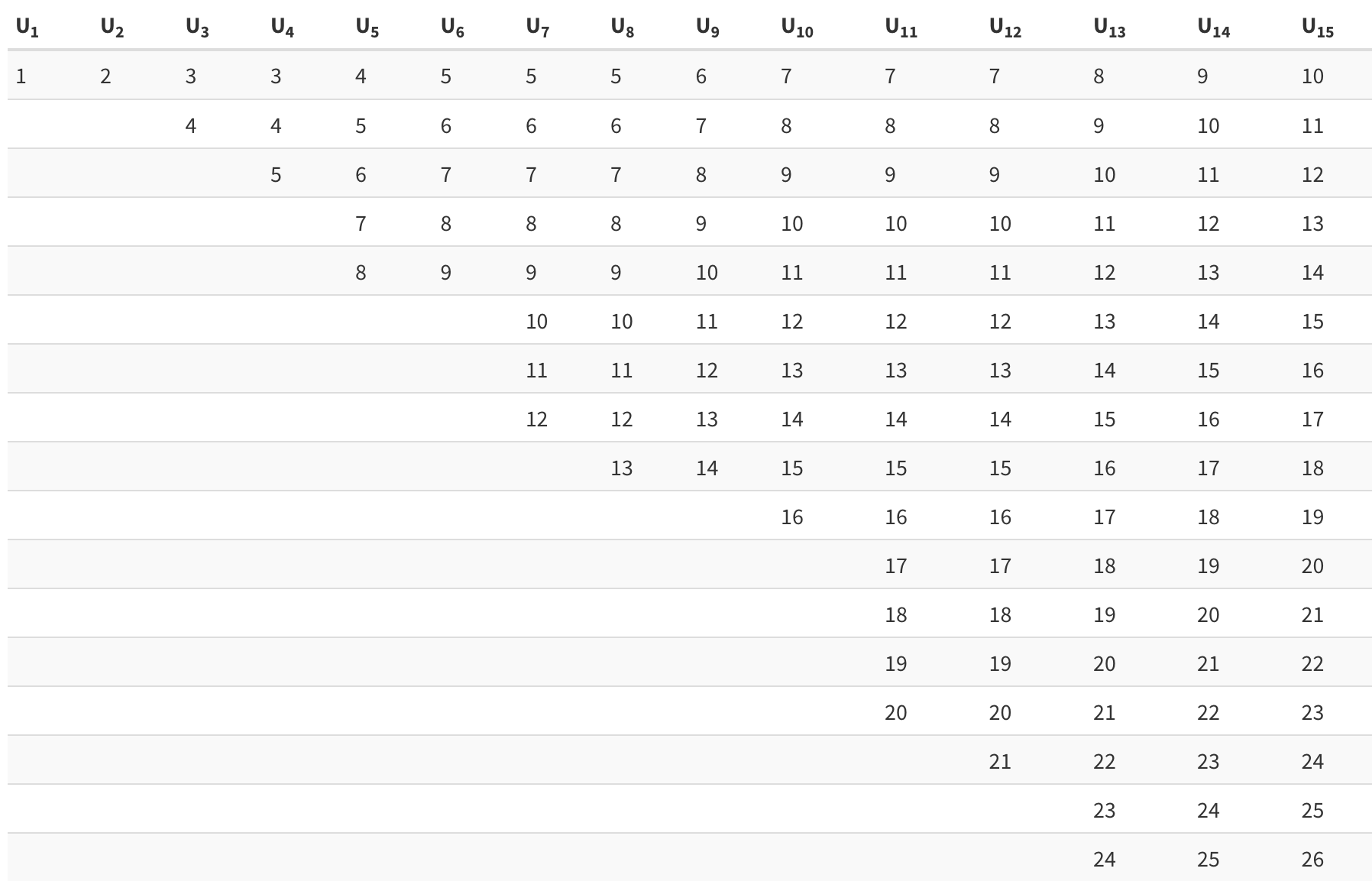}
			\caption{$U_k$ for $k \in \{1, \dots, 15\}$.}
			\label{fig:U_k}
		\end{table}

	\end{example}
	\medskip
	
	\begin{example}
		Let $\{p_n\}$ be an enumeration of the prime numbers, and consider the Puiseux monoid $P = \big\langle 1/p_n : n \in \mathbb{N} \big\rangle$. It is not difficult to argue that $P$ is atomic with $\mathcal{A}(P) = \{1/p_n : n \in \mathbb{N}\}$. As $\mathcal{A}(P)$ is a bounded subset of positive rationals, the Puiseux monoid $P$ is bounded. Notice, however, that $0$ is a limit point of $P$. By Proposition~\ref{prop:union of sets of lengths: infinite case}, it follows that the local elasticities $\rho_k(P)$ are infinite for all $k$ sufficiently large. 
	\end{example}
	\medskip

	The condition of boundedness on Proposition~\ref{prop:union of sets of lengths: finite case} is also required, as shown by the following proposition.
	
	\begin{prop} \label{prop:PM with all its local elasticities infinite}
		There exist infinitely many non-isomorphic Puiseux monoids without $0$ as a limit point that have no finite local elasticities.
	\end{prop}
	
	\begin{proof}
		Let $\mathcal{P} = \{S_n : n \in \mathbb{N} \}$ be a family of disjoint infinite sets of odd prime numbers. For each set $S_n$, we will construct an atomic Puiseux monoid $M_n$. Then we will show that $M_i \cong M_j$ implies $i = j$.
		
		Fix $j \in \mathbb{N}$ and take $p \in S_j$. To construct the Puiseux monoid $M_j$, let us inductively create a sequence $\{A_n\}_{n \in \mathbb{N}}$ of finite subsets of positive rationals with $A_1 \subsetneq A_2 \subsetneq \cdots$ such that, for each $k \in \mathbb{N}$, the following three conditions hold:
		\begin{enumerate}
			\item $\mathsf{d}(A_k)$ consists of odd prime numbers; \vspace{2pt}
			\item $\mathsf{d}(\max A_k) = \max \, \mathsf{d}(A_k)$; \vspace{2pt}
			\item $A_k$ minimally generates the Puiseux monoid $P_k = \langle A_k \rangle$.
		\end{enumerate}
		Take $A_1 = \{1/p\}$, with $p$ an odd prime number, and assume we have already constructed the sets $A_1, \dots, A_n$ for some $n \in \mathbb{N}$ satisfying our three conditions. To construct $A_{n+1}$, we take $a = \max A_n$ and let
		\[
			b_1 = \frac{\mathsf{n}(a) \lfloor q/2 \rfloor}{q} \ \text{ and } \ b_2 = \frac{\mathsf{n}(a)\big(q - \lfloor q/2 \rfloor \big)}{q},
		\]
		where $q$ is an odd prime in $S_j$ satisfying $q > \max \mathsf{d}(A_n)$ and $q \nmid \mathsf{n}(a)$. Using the fact that $q \ge 5$ and $\mathsf{d}(a) \ge 3$, one obtains that
		\[
			b_2 > b_1 = \frac{\lfloor q/2 \rfloor}{q} \mathsf{n}(a) > \frac 13 \mathsf{n}(a) \ge a.	
		\]
		Now set $A_{n+1} = A_n \cup \{b_1, b_2\}$. Notice that $b_1 + b_2 = \mathsf{n}(a)$. Clearly, $A_n \subsetneq A_{n+1}$, and condition~(1) is an immediate consequence of our inductive construction. In addition,
		\[
			\mathsf{d}(\max A_{n+1}) = \mathsf{d}(b_2) = q = \max \mathsf{d}(A_{n+1}),
		\]
		which is condition (2). Therefore it suffices to verify that $A_{n+1}$ minimally generates $P_{n+1} = \langle A_{n+1} \rangle$. Because both $b_1$ and $b_2$ are greater than every element in $A_n$, we only need to check that $b_1 \notin P_n$ and $b_2 \notin \langle A_n \cup \{b_1\} \rangle$. Let $d$ be the product of all the elements in $\mathsf{d}(A_n)$. Assuming that $b_1 = a_1 + \dots + a_r$ for some $a_1, \dots, a_r \in A_n$, and multiplying both sides of the same equality by $qd$, we would obtain that $q \mid \mathsf{n}(b_1)$, which contradicts that $q \nmid \mathsf{n}(a)$. Hence $b_1 \notin P_n$. Similarly, one finds that $b_2 \notin P_n$. Suppose, again by contradiction, that $b_2 \in \langle A_n \cup \{b_1\} \rangle$. Then there exist $a'_1 , \dots, a'_s \in A_n$ and $m \in \mathbb{N}$ such that $b_2 = mb_1 + a'_1 + \dots + a'_s$. Notice that $2b_1 = \mathsf{n}(a) (q-1)/q > b_2$, which implies that $m \le 1$. As $b_2 \notin P_n$, it follows that $m=1$. Then we can write
		\begin{align} \label{eq:b_2}
			\frac{\mathsf{n}(a)}q = b_2 - b_1 = \sum_{i=1}^s a'_i. 
		\end{align}
		Once again, we can multiply the extreme parts of the equality~(\ref{eq:b_2}) by $q \, \mathsf{d}(\{a'_1, \dots, a'_s\})$, to obtain that $q \mid \mathsf{n}(a)$, a contradiction. As a result, condition~(3) follows.
		
		Now set $M_j := \cup_{n \in \mathbb{N}} P_n$. As $P_1 \subsetneq P_2 \subsetneq \dots$, the set $M_j$ is, indeed, a Puiseux monoid. We can easily see that $M_j$ is generated by the set $A := \cup_{n \in \mathbb{N}} A_n$. Let us verify now that $\mathcal{A}(M_j) = A$. It is clear that $\mathcal{A}(M_j) \subseteq A$. To check the reverse inclusion, suppose that $a \in A$ is the sum of atoms $a_1, \dots, a_r \in \mathcal{A}(M_j)$. Take $t \in \mathbb{N}$ such that $a, a_1, \dots, a_r \in A_t$. Because $A_t$ minimally generates $P_t$ it follows that $r=1$ and $a = a_1$ and, therefore, that $a \in \mathcal{A}(M_j)$. Hence $\mathcal{A}(M_j) = A$, which implies that $M_j$ is an atomic monoid.
		
		To disregard $0$ as a limit point of $M_j$, it is enough to observe that $\min \mathcal{A}(M_j) = 1/p$. We need to show then that $\rho_k(M_j) = \infty$ for $k \ge 2$. Set $a_n = \max A_n$. When constructing the sequence $\{A_n\}$, we observed that $\mathsf{n}(a_n) = b_{n_1} + b_{n_2}$, where $\{b_{n_1}, b_{n_2}\} = A_{n+1} \setminus A_n$. Because $\mathsf{n}(a_n) \in M_j$ and
		\[
			b_{n_1} + b_{n_2} =\mathsf{n}(a_n) = \mathsf{d}(a_n) a_n,
		\]
		one has that the factorizations $z = b_{n_1} + b_{n_2}$ and $z' = \mathsf{d}(a_n) a_n$ are both in $\mathsf{Z}(\mathsf{n}(a_n))$. Since $|z| = 2$ and $|z'| = \mathsf{d}(a_n)$ it follows that $\mathsf{d}(a_n) \in \mathcal{U}_2(M_j)$. By condition~(2) above, $\mathsf{d}(a_n) =\mathsf{d}(\max A_n) = \max \mathsf{d}(A_n)$. This implies that the set $\{\mathsf{d}(a_n) : n \in \mathbb{N} \}$ contains infinitely many elements. As $\{\mathsf{d}(a_n) : n \in \mathbb{N} \} \subseteq \mathcal{U}_2(M_j)$, we obtain that $\rho_2(M_j) = \infty$. Hence $\rho_k(M_j) = \infty$ for all $k \ge 2$.
		
		We have just constructed an infinite family $\mathcal{F} := \{M_n : n \in \mathbb{N}\}$ of atomic Puiseux monoids with infinite $k$-elasticities. Let us show now that the monoids in $\mathcal{F}$ are pairwise non-isomorphic. To do this we use the fact that the only homomorphisms between Puiseux monoids are given by rational multiplication \cite[Lemma~3.3]{GGP16}. Take $i,j \in \mathbb{N}$ such that $M_i \cong M_j$. Then there exists $r \in \mathbb{Q}$ such that $M_i = rM_j$. Let $m \in M_j$ such that $\mathsf{d}(m) = p$ and $p \nmid \mathsf{n}(r)$ for some prime $p$ in $S_j$. Since the element $rm \in M_i$ and $p \mid \mathsf{d}(rm)$, we must have that the prime $p$ belongs to $S_i$. Because the sets in $\mathcal{P}$ are pairwise disjoint, we conclude that $i = j$. This completes the proof.
	\end{proof}
	\medskip

	Proposition~\ref{prop:union of sets of lengths: infinite case} (respectively, Proposition~\ref{prop:union of sets of lengths: finite case}) establishes sufficient conditions under which a Puiseux monoid has most of its local elasticities infinite (respectively, finite). In addition, we have verified that such conditions are not necessary. For the sake of completeness, we now exhibit a Puiseux monoid that does not satisfy the conditions of either of the propositions above and has no finite $k$-elasticity for any $k \ge 2$.
	
	\begin{example}
		Consider the Puiseux monoid
		\[
			P = \left\langle \left(\frac{2}{3} \right)^n : \, n \in \mathbb{N} \right \rangle.
		\]
		It was proved in \cite[Theorem~6.2]{GG17} that $P$ is atomic and $\mathcal{A}(P) = \{(2/3)^n : n \in \mathbb{N}\}$. In addition, it is clear that $P$ is bounded, has $0$ as a limit point, and does not contain any stable atoms. So neither Proposition~\ref{prop:union of sets of lengths: infinite case} nor Proposition~\ref{prop:union of sets of lengths: finite case} applies to $P$. Now we argue that $\rho_k(P) = \infty$ for each $k \in \mathbb{N}$ such that $k \ge 2$.
		
		Take $k \ge 2$ and set $x = k\frac{2}{3} \in P$. Notice that, by definition, $x \in \mathsf{L}^{-1}(k)$. We can conveniently rewrite $x$ as
		\[
			x = \big((k - 2) + 2\big) \frac{2}{3} = (k - 2)\frac{2}{3} + 3\cdot  \left(\frac{2}{3}\right)^2\! \!,
		\]
		which reveals that $z = (k-2)\frac 23 + 3(\frac 23)^2$ is a factorization of $x$ with $|z| = k+1$. Taking $k' = 3$ to play the role of $k$ and repeating this process as many times as needed, one can obtain factorizations of $x$ of lengths as large as one desires. The fact that $k$ was chosen arbitrarily implies now that $\rho_k(P) = \infty$ for each $k \ge 2$. \\
	\end{example}
	\medskip

	\section{The Primary Case} \label{sec:primary case}
	
	Recall that a Puiseux monoid is said to be primary if it can be generated by a subset of rational numbers whose denominators are pairwise distinct primes. In Proposition~\ref{prop:union of sets of lengths: finite case}, we established a sufficient condition on Puiseux monoids to ensure that all their local $k$-elasticities are finite. Here we restrict our study to the case of primary Puiseux monoids, providing two more sufficient conditions to guarantee the finiteness of all the local $k$-elasticities.
	
	\begin{theorem}\label{theo:sufficient conditions for finite elasticity in ppm}
		For a primary Puiseux monoid $P$, the following two conditions hold.
		\begin{enumerate}
			\item If $0$ is not a limit point of $P$, then $\rho_k(P) < \infty$ for every $k \in \mathbb{N}$. \vspace{3pt}
			\item If $P$ is bounded and has no stable atoms, then $\rho_k(P) < \infty$ for every $k \in \mathbb{N}$.
		\end{enumerate}
	\end{theorem}
	
	\begin{proof}
		Because every finitely generated Puiseux monoid is isomorphic to a numerical semigroup, and numerical semigroups have finite $k$-elasticities, we can assume, without loss of generality, that $P$ is not finitely generated.
		
		To prove condition~(1), suppose, by way of contradiction, that $\rho_k(P) = \infty$ for some $k \in \mathbb{N}$. Because $0$ is not a limit point of $P$ there exists $q \in \mathbb{Q}$ such that $0 < q < a$ for each $a \in \mathcal{A}(P)$. Let
		\[
			\ell = \min \{n \in \mathbb{N} : |\mathcal{U}_n(P)| = \infty\}.
		\]
		Clearly, $\ell \ge 2$. Let $m = \max \, \mathcal{U}_{\ell - 1}(P)$. Now take $N \in \mathbb{N}$ sufficiently large such that, for each $a \in \mathcal{A}(P)$, $a > N$ implies that $\mathsf{d}(a) > \ell$. As $\mathcal{U}_\ell(P)$ contains infinitely many elements, there exists $k \in \mathcal{U}_\ell(P)$ such that
		\[
			k > \max\bigg\{\frac{\ell}{q}N, \, m + 1 \bigg\}.
		\]
		In particular, $k-1$ is a strict upper bound for $\mathcal{U}_{\ell - 1}(P)$. As $k \in \mathcal{U}_\ell(P)$, we can choose an element $x \in P$ such that $\{k,\ell\} \subseteq \mathsf{L}(x)$. Take $A = \{a_1, \dots, a_k\} \subsetneq \mathcal{A}(P)$ and $B = \{b_1, \dots, b_\ell\} \subsetneq \mathcal{A}(P)$ with
		\begin{align} \label{eq:different length factorizations 1}
			a_1 + \dots + a_k = x = b_1 + \dots + b_\ell.
		\end{align}
		Observe that the sets $A$ and $B$ must be disjoint, for if $a \in A \cap B$, canceling $a$ in (\ref{eq:different length factorizations 1}) would yield that $\{\ell - 1, k - 1\} \subseteq \mathsf{L}(x - a)$, which contradicts that $k-1$ is a strict upper bound for $\mathcal{U}_{\ell - 1}(P)$. Because $k > (\ell/q)N$, it follows that
		\[
			x > kq > \ell N.
		\]
		Therefore $b := \max\{b_1, \dots, b_\ell\} > N$, which implies that $p = \mathsf{d}(b) > \ell$. Since $a_i \neq b$ for each $i = 1, \dots, k$, it follows that $p \notin \mathsf{d}(\{a_1, \dots, a_k\})$. We can assume, without loss of generality, that there exists $j \in \{1, \dots, \ell\}$ such that $b_i \neq b$ for every $i \le j$ and $b_{j+1} = \dots = b_\ell = b$. This allows us to rewrite (\ref{eq:different length factorizations 1}) as
		\begin{align} \label{eq:different length factorization 2}
			(\ell - j)b = \sum_{i=1}^k a_i - \sum_{i=1}^j b_i.
		\end{align}
		After multiplying \ref{eq:different length factorization 2} by $p$ times the product $d$ of all the denominators of the atoms $\{a_1, \dots, a_k, b_1, \dots, b_j\}$, we find that $p$ divides $d(\ell - j)b$. As $\gcd(p,d) = 1$ and $\ell - j < p$, it follows that $p$ divides $\mathsf{n}(b)$, which is a contradiction. Hence we conclude that $\rho_k(P) < \infty$ for every $k \in \mathbb{N}$. \vspace{4pt}
		
		Now we argue the second condition. Let $\{a_n\}$ be an enumeration of the elements of $\mathcal{A}(P)$ such that $\{\mathsf{d}(a_n)\}$ is an increasing sequence. Set $p_n = \mathsf{d}(a_n)$. Since $P$ has no stable atoms, $\lim \mathsf{n}(a_n) = \infty$. Let $B$ be an upper bound for $\mathcal{A}(P)$.
		
		Suppose, by way of contradiction, that $\rho_n(P) = \infty$ for some $n \in \mathbb{N}$. Let $k$ be the smallest natural number such that $|\mathcal{U}_k(P)| = \infty$. Now take $\ell \in \mathcal{U}_k(P)$ large enough such that $\ell - 1 > \max \, \mathcal{U}_{k-1}(P)$ and for each $a \in \mathcal{A}(P)$ satisfying $a \le Bk/\ell$ we have that $\mathsf{n}(a) > Bk$. Take $x \in \mathsf{L}^{-1}(k)$ such that $a_1 + \dots + a_k = x = b_1 + \dots + b_\ell$ for some $a_1, \dots, a_k, b_1, \dots, b_\ell \in \mathcal{A}(P)$. Now set $b = \min\{b_1, \dots, b_\ell\}$. Then
		\[
			b \le \frac{b_1 + \dots + b_\ell}{\ell} = \frac{a_1 + \dots + a_k}{\ell} \le \frac{Bk}{\ell}.
		\]
		Therefore $\mathsf{n}(b) > Bk$. We claim that $\mathsf{d}(b) \notin \mathsf{d}(\{a_1, \dots, a_k\})$. Suppose by contradiction that this is not the case. Then $b = a_i$ for some $i \in \{1, \dots, k\}$. This implies that $\{k - 1, \ell - 1 \} \subseteq \mathsf{L}(x-b)$, contradicting that $\ell - 1 > \max \, \mathcal{U}_{k-1}(P)$. Hence $\mathsf{d}(b) \notin \mathsf{d}(\{a_1, \dots, a_k\})$. Now assume, without loss of generality, that there exists $j \in \{1, \dots, \ell\}$ such that $b_i \neq b$ for each $i \le j$ and $b_{j+1} = \dots = b_\ell = b$. Write
		\begin{align} \label{eq:different length factorization 3}
			(\ell - j) b = \sum_{i=1}^k a_i - \sum_{i=1}^j b_i.
		\end{align}
		From (\ref{eq:different length factorization 3}) we obtain that $p_\ell$ divides $\ell - j$. As a consequence,
		\[
			Bk \ge \sum_{i=1}^k a_i \ge \frac{\ell - j}{p_\ell} \mathsf{n}(b) \ge \mathsf{n}(b) > Bk,
		\]
		which is a contradiction. Hence $\rho_k(P) < \infty$ for every $k \in \mathbb{N}$.
	\end{proof}
	\medskip		

	The sufficient conditions in part (1) of Theorem~\ref{theo:sufficient conditions for finite elasticity in ppm}(1) and the condition of boundedness in part (2) of Theorem~\ref{theo:sufficient conditions for finite elasticity in ppm} are not necessary, as the following example illustrates.
	
	\begin{example} 
		\ \\[-1em]
		\begin{enumerate}
			\item Consider the primary Puiseux monoid 
			\[
				P = \left \langle \frac{n}{p_n} : n \in \mathbb{N} \right \rangle,
			\]
			where $\{p_n\}$ is the increasing sequence of all prime numbers. Since $\mathcal{A}(P) = \{n/p_n : n \in \mathbb{N}\}$, it follows that $P$ does not contain any stable atom. It is well known that the sequence $\{n/p_n\}$ converges to $0$, which implies that $P$ is bounded. Hence part~(2) of Theorem~\ref{theo:sufficient conditions for finite elasticity in ppm} ensures that $\rho_k(P) < \infty$ for all $k \in \mathbb{N}$. Thus, the reverse implication of part~(1) in Theorem~\ref{theo:sufficient conditions for finite elasticity in ppm} does not hold. \vspace{4pt}
			
			\item Consider now the Puiseux monoid
			\[
				P = \left \langle \frac{p_n^2 - 1}{p_n} : n \in \mathbb{N} \right \rangle,
			\]
			where $\{p_n\}$ is any enumeration of the prime numbers. Since $0$ is not a limit point of $P$, we can apply part~(1) of Theorem~\ref{theo:sufficient conditions for finite elasticity in ppm} to conclude that $\rho_k(P) < \infty$ for all $k \in \mathbb{N}$. Notice, however, that $P$ is not bounded. Therefore, the boundedness in part~(2) of Theorem \ref{theo:sufficient conditions for finite elasticity in ppm} is not a necessary condition.
		\end{enumerate}	
	\end{example}
	\medskip
			
	\section{Acknowledgments}
	
		I would like to thank Salvatore Tringali for proposing some of the questions motivating this project and Felix Gotti for many enlightening conversations about Puiseux monoids. Finally, I would like to thank the anonymous referee, whose helpful suggestions lead to an improvement of the last version of this paper.

\end{document}